\documentclass[reqno]{amsart}

\usepackage{amsmath}
\usepackage{amssymb}
\usepackage{amscd}
\usepackage{amsfonts}
\usepackage{amsthm}
\usepackage{color}

\input xy
\xyoption{all}

\theoremstyle{definition}
\newtheorem{thm}{Theorem}
\newtheorem*{thm*}{Theorem}
\newtheorem{lem}[thm]{Lemma}
\newtheorem{corollary}[thm]{Corollary}
\newtheorem{prop}[thm]{Proposition}
\newtheorem{defn}[thm]{Definition}
\newtheorem{remark}[thm]{Remark}
\newtheorem{exm}[thm]{Example}
\numberwithin{equation}{section}
\numberwithin{thm}{section}

\newcommand{\ft}{{\mathfrak{t}}}

\newcommand{\fh}{{\mathfrak{h}}}
\newcommand{\fk}{{\mathfrak{k}}}
\newcommand{\fs}{{\mathfrak{s}}}
\newcommand{\fq}{{\mathfrak{q}}}
\newcommand{\gd}{\mathfrak{t}^{*}}

\renewcommand{\S}{{\mathbb{S}}}
\newcommand{\R}{\mathbb{R}}
\newcommand{\Z}{{\mathbb{Z}}}
\newcommand{\C}{{\mathbb{C}}}
\newcommand{\cA}{{\mathcal{A}}}
\newcommand{\cK}{{\mathcal{K}}}
\newcommand{\cP}{{\mathcal{P}}}

\newcommand{\spass}{{cohomological assignment}}
\newcommand{\quotmod}{{defect}}
\newcommand{\img}{{\mathcal{H}}_T(M)}

\newcommand{\Arktm}[4]{\cA_{#3}^{#2}(#4_{(#1)})}
\newcommand{\Aktm}[3]{\cA_{#2}^{#1}(#3)}

\newcommand{\purge}[1]{} 

\DeclareMathOperator{\maps}{Maps}
\DeclareMathOperator{\Stab}{Stab}

\DeclareMathOperator{\im}{im}

\title[Assignments]{Polynomial Assignments}
\author[V. Guillemin]{Victor Guillemin}
\address{Department of Mathematics, MIT, Cambridge, MA 02139}
\email{vwg@math.mit.edu}

\author[S. Sabatini]{Silvia Sabatini}
\address{Department of Mathematics, EPFL, Lausanne, Switzerland}
\email{silvia.sabatini@epfl.ch}

\author[C. Zara]{Catalin Zara}
\address{Department of Mathematics, University of Massachusetts
Boston, MA 02125}
\email{catalin.zara@umb.edu}

\date{April 18, 2013}

\begin{document}

\begin{abstract}
The concept of \emph{assignments} was introduced in \cite{GGK} as a method for
extracting geometric information about group actions on manifolds from
combinatorial data encoded in the infinitesimal orbit-type stratification. In this
paper we will answer in the affirmative a question posed in \cite{GGK} by
showing that the equivariant cohomology ring of $M$ is to a large extent
determined by this data.


\end{abstract}

\maketitle

\tableofcontents
\section{Introduction}
\label{sec:intro}

Let $T$ be an $r$-dimensional torus, where $r\geqslant 2$, and $\ft$ its Lie algebra. Given a compact manifold $M$ and an effective action of $T$ on $M$ we will denote by $\ft_p$ the Lie algebra of the stabilizer group $T_p$ 
of $p \in M$. The set of points $q \in M$ with $\ft_q=\ft_p$ is a submanifold of $M$ and the connected 
component $X$ of this set containing $p$ is the \emph{infinitesimal orbit type stratum of } $M$ 
\emph{ through } $p$. These strata form a poset $\mathcal{P}_T(M)$ with respect to the ordering
$$X \preccurlyeq Y \Leftrightarrow X \subseteq \overline{Y}$$
and for each $X \in \mathcal{P}_T(M)$ we will denote by $\ft_X$ the Lie algebra $\ft_p$ for $p \in X$.
Note that  $X\preccurlyeq Y$ implies $\ft_Y\subseteq \ft_X$.

The objects we will be interested in this paper are functions $\phi$ on $\mathcal{P}_T(M)$ that assign to each stratum $X$ a map
\begin{equation}
\phi_X \colon \ft_X \to \R
\end{equation}
such that 
\begin{equation}
\left. \phi_Y = \phi_X\right|_{\ft_Y}
\end{equation}
for  all $X \preccurlyeq Y$. Following \cite{GGK} we will call such functions \emph{assignments} 
and in particular we will call $\phi$ a \emph{polynomial assignment} if the map $\phi_X$ is a 
polynomial map, \emph{i.e.} an element of $\S(\ft_X^*)$ for every $X \in \mathcal{P}_T(M)$. 
These polynomial assignments form a graded ring $\Aktm{}{T}{M}$ and in \cite{GGK} it was 
conjectured that there exists a relation between this ring and 
the equivariant cohomology ring $H_T^*(M) :=H_T^*(M;\R)$ of $M$ (with the caveat, however, 
that ``the nature and explicit form of this relation is as yet unclear to the authors.") The goal 
of this paper is to describe some results that confirm this conjecture for a large class of spaces. 
More explicitly we will show that there is a canonical map of graded rings
\begin{equation}\label{eq:0.3}
\gamma \colon H^{even}_T(M) \to \Aktm{}{T}{M}
\end{equation}
and will show that if $M$ is equivariantly formal and $M^T$ is finite, then this map is injective. 
Note that under these conditions $H_T^{odd}(M)$ is zero (see Proposition \ref{odd cohom}).

In Section~\ref{sec:cas} we will give a combinatorial characterization of the image of  the map \eqref{eq:0.3}. The key ingredient in this characterization is the Atiyah-Bott-Berline-Vergne Localization Theorem in equivariant cohomology: Recall that in the Cartan model an equivariant cohomology class $[\eta]$ is represented by a sum 
\begin{equation}
\sum \eta_i \otimes f_i \; ,
\end{equation}
with $\eta_i \in \Omega(M)^T$, $f_i \in \S(\ft^*)$ with $\deg \eta_i + \deg f_i = k$, and if $M^T$ is finite the cohomological assignment associated with $[\eta]$ is defined on strata of dimension zero, i.e. points $p \in M^T$ by the recipe
\begin{equation}
M^T \ni p \to \gamma([\eta])(p) = \sum (\iota_p^* \eta_i) f_i   \; .
\end{equation}
However, by Atiyah-Bott-Berline-Vergne (ABBV), if $M$ is a Hamiltonian $T$-space, the integral of $[\eta]$ over a closed $T$-invariant symplectic submanifold $X$ of $M$ 
is given by the localization formula
\begin{equation}\label{eq:abbv}
\sum_{p \in X^T} \frac{\gamma([\eta])(p)}{\prod \alpha_{j,p}} \in \S^{k-\ell}(\ft^*)
\end{equation}
where $2\ell = \dim X$ and the $\alpha_{j,p}$'s are the weights of the isotropy representation of $T$ on the tangent space to $X$ at $p$. These identities impose an (unmanageably large) number of constraints on a polynomial assignment for it to be a cohomological assignment and in Section~\ref{sec:cas} we will analyze these constraints and show that a (manageably small) subset of them are sufficient as well as necessary for a polynomial assignment to be cohomological. 
To describe this manageably small subset we note that by the Chang-Skjelbred theorem (see \cite{CS})
the image of $H_T(M)$ in $H_T(M^T)$ is the intersection of the images of the maps
$H_T(M^K)\to H_T(M^T)$, where $K$ is a codimension one subtorus of $T$.
As a consequence of this fact we will show that the image of \eqref{eq:0.3} is characterized by the
ABBV conditions on submanifolds $X$ in the closures of the degree one strata of the infinitesimal orbitype stratification.
These conditions turn out to be easy to analyze since the $\alpha_{j,p}$'s in the denominator of
\eqref{eq:abbv} are all elements of the one dimensional vector space $\mathfrak{t}_X^{\perp}$, i.e.
if we represent $H_T(X)$ as a tensor product, 
$$
H_T(X)=H_{S^1}(X)\otimes \mathbb{S}(\mathfrak{t}_X^{\perp})
$$
they just become conditions on $H_{S^1}(X)$.

If all the degree one strata are of dimension less than or equal to $2$, these conditions reduce to the
``Goresky-Kottwitz-MacPherson conditions" and we recover a well known result of GKM (\cite{GKM}).
\begin{thm}
If all the orbitype strata of degree one are of dimension less than or equal to two, the map
$$
H_T(M)\to \mathcal{A}_T(M)
$$
is bijective.
\end{thm}

Finally in Section~\ref{sec:kirwan} we will show that the assignment ring, like the 
equivariant cohomology ring, has nice functorial properties with respect to morphisms 
between $T$-manifolds. This enables one to study from the assignment perspective various functorial operations in equivariant DeRham theory, such as fiber integration, Kirwan maps, and flops  - topics that we hope to pursue in more detail in a sequel to this paper.

We would like to thank Yael Karshon for her participation in a series of skype sessions in 
which, with her help, preliminary versions of the results above were hammered out, and 
Tara Holm for informing us about results similar to the results 
in Section~\ref{sec:cas} in the thesis of her student Milena Pabiniak.
We would also like to thank Matthias Franz for interesting suggestions about
potential generalizations of this paper.
 Last but not least, 
we would like to thank Sue Tolman for reading close-to-last versions of our manuscript 
and making valuable suggestions about how to improve it.

\section{Polynomial Assignments}
\label{sec:assignments}
Let $T$ be a (real) torus of dimension $r \geqslant 2$ acting smoothly and effectively on a compact
manifold $M$ and $\ft$ the Lie algebra of $T$. Recall that
an action is called \emph{effective} if for every $p\in M$ there exists $a\in T$ such that $a\cdot p\neq p$.
Henceforth, for simplicity we will assume $M$ to be orientable and oriented. 

For a point $p$ in $M$, the \emph{stabilizer} of $p$ is the subgroup 
$$\Stab(p) = \{ a \in T \, | \, a\cdot p = p\}\; ,$$
and the \emph{infinitesimal stabilizer} of $p$ is the Lie algebra of $\Stab(p)$, 
$$\ft_p = \{ \xi \in \ft \; | \; \xi_M(p) = 0\}\; $$
where $\xi_M$ is the vector field 
\begin{equation}\label{eq:xiM}
\xi_M(p) = \Bigl. \frac{d}{dt} \Bigr|_{t=0} (\exp(t\xi) \cdot p)\; .
\end{equation}

For a Lie subalgebra $\fh \subseteq \ft$, let $\cP_T(M, \fh)$ be the set of 
connected components of 
$$M^{\fh} = \{p \in M \, | \, \ft_p = \fh\} \; ;$$
these connected components are smooth submanifolds of $M$ and are called 
the \emph{infinitesimal orbit type strata} of $M$ or, shorter, just 
\emph{infinitesimal strata} of $M$. For such a stratum $X$, 
let $\ft_X$ be the common infinitesimal stabilizer of all points in $X$ 
and $T_X=\exp(\ft_X) \subset T$ the subtorus of $T$ with Lie algebra $\ft_X$;
then $X$ is open and dense in a connected component of $M^{T_X}$. The \emph{degree} 
of the stratum $X$ is the codimension of the infinitesimal stabilizer $\ft_X$ in $\ft$.
For example, the zero-degree strata are the fixed points.

For $0 \leqslant j \leqslant r$ let $\cP_T^j(M)$ be the set of infinitesimal strata 
of degree at most~$j$~and 
\begin{equation}
M_{(j)} = \{ p \in M \mid \dim \ft_p \geqslant r- j \}
\end{equation}
the \emph{infinitesimal $j$-skeleton} of $M$; then $M_{(j)} $ is the union 
of all infinitesimal strata in $\cP_T^j(M)$. Let $\cP_T(M) = \cP_T^r(M)$ be the set of all 
infinitesimal strata. Then
\begin{equation}\label{eq:skeletons}
M^T = \cP_T^0(M) \subseteq \cP_T^1(M) \subseteq \dotsb \subseteq 
\cP_T^r(M) = \cP_T(M)\; .
\end{equation}

The set $\cP_T(M)$ has a natural partial 
order: $X \preccurlyeq Y$ if and only if the stratum  $X$ is contained 
in $\overline{Y}$, the closure of the stratum $Y$. If that is the case, then 
$\ft_Y \subseteq \ft_X$, and this inclusion $i_{\ft_Y}^{\ft_X}$  induces a natural 
projection $\pi_Y^X = (i_{\ft_Y}^{\ft_X})^* \colon \ft_X^* \to \ft_Y^*$ for the dual spaces. 
This projection extends to a degree-preserving map 
$$\pi_Y^X \colon \S(\ft_X^*) \to \S(\ft_Y^*)$$
from the symmetric algebra of polynomial functions of $\ft_X$ to the 
symmetric algebra of polynomial functions on $\ft_Y$:
if $f \in \S(\ft_X^*)$ is a polynomial function on $\ft_X$, then 
$\pi_Y^X(f) \in \S(\ft_Y^*)$ is the restriction of $f$ to $\ft_Y$. 
We consider that elements of $\ft^*$ have degree two, hence a 
homogeneous polynomial of degree $k$ corresponds to a homogeneous 
element of degree $2k$ of the symmetric algebra of $\ft^*$.
We denote the space of homogeneous polynomials of degree $k$ by $\mathbb{S}^k(\gd)$.

\begin{defn} A \emph{polynomial assignment} of degree $2k$ on $\cP_T^j(M)$ 
is a function $A$ that associates to 
each infinitesimal stratum $X$ of degree at most $j$ a homogeneous polynomial 
$A(X)$ of degree $k$ on $\ft_X$, such that if $X \preccurlyeq Y$, then 
$A(Y) = \pi_Y^X (A(X))$.
We denote the set of all such assignments by $\Arktm{j}{2k}{T}{M}$.
\end{defn}

 Let $\Arktm{j}{}{T}{M}$ be the graded $\R$-algebra
\begin{equation}
\Arktm{j}{}{T}{M} = \bigoplus_{k} \Arktm{j}{2k}{T}{M}\; 
\end{equation}
and $\Aktm{}{T}{M} = \Arktm{r}{}{T}{M}$ the graded $\R$-algebra of all polynomial 
assignments on $M$. The inclusions \eqref{eq:skeletons} induce degree-preserving restriction 
maps
\begin{equation}
\Arktm{j}{}{T}{M} \to \Arktm{\ell}{}{T}{M}
\end{equation}
for all $0 \leqslant \ell \leqslant j \leqslant r$.

A polynomial assignment on $M$ is determined by the values it takes on the minimal strata, 
that is, the strata $X$ which are minimal under the $\preccurlyeq$ ordering. 
A \emph{minimal stratum assignment} is a function $A$ that associates to each 
minimal stratum $X$ an element $A(X) \in \S(\ft_X^*)$, such that if $X$ and $Y$ are 
minimal strata and $Z$ is a stratum such that $X, Y \preccurlyeq Z$, then $\pi_{Z}^{X}(A(X)) = \pi_{Z}^{Y}(A(Y))$
and $A(Z)$ is defined to be
$$A(Z):=\pi_{Z}^{X}(A(X))\,.$$ 
Hence any minimal stratum assignment extends uniquely to a polynomial assignment. 
In particular, if the closure of every infinitesimal stratum contains a fixed point, then the restriction
$$\Arktm{j}{}{T}{M} \to \Arktm{0}{}{T}{M}$$
is injective for all $0 \leqslant j \leqslant r$.

We conclude the section with a few simple examples.

\begin{exm}\label{ex:cp1}
Let $T= (S^1)^r$ be an $r$-dimensional torus, with Lie algebra $\ft = \R^r$, and exponential map
$\exp \colon \R^r \to T$, 
$$\exp(\xi_1, \ldots, \xi_r) = (e^{2\pi i \xi_1}, \ldots, e^{2\pi i \xi_r})\; .$$
Let $\alpha \in \ft^*$ such that $\alpha (\Z^r) \subseteq \Z$ and consider the action of $T$ on 
$$\C P^1 = \{ [z_0:z_1]\; | \; (z_0,z_1) \in \C^2\setminus 0 \}$$
given by 
$$\exp(\xi) \cdot [z_0:z_1] = [e^{2\pi i \alpha(\xi)} z_0 : z_1]\; .$$

The points $p_1= [0:1]$ and $p_2 = [1:0]$ are fixed by $T$, hence have infinitesimal stabilizer $\ft$. 
For all other points, the infinitesimal stabilizer is $\ft_{\alpha} = \ker(\alpha)$. There are three infinitesimal orbit type strata, 
$$X_1 = \{p_1\}, \quad X_2 = \{ p_2\}, \quad \text{ and } \quad X_0 = \{ [z_0:z_1]\; | \; z_0z_1 \neq 0\}\; .$$

A polynomial assignment is a function $A$ such that
$$A(X_1) = f_1 \in \S(\ft^*), \quad A(X_2) = f_2 \in \S(\ft^*)$$
and
$$A(X_0) = f_0 \in \S(\ft_{\alpha}^*) = \S((\ker(\alpha))^*) \simeq \S(\ft^*/\R\alpha)\; .$$
Moreover, the partial order relations $X_1 \preccurlyeq X_0$ and $X_2 \preccurlyeq X_0$ 
impose the conditions
$$f_0 = \pi_{X_0}^{X_1}(f_1), \quad \text{ and } \quad f_0 = \pi_{X_0}^{X_2}(f_2)\; .$$
Then $A$ is determined by $f_1$ and $f_2$, which must satisfy the compatibility condition 
$$\pi_{X_0}^{X_1}(f_1) = \pi_{X_0}^{X_2}(f_2) \Longleftrightarrow f_1-f_2 \in \alpha \S(\ft^*)\; .$$
\end{exm}

This is the GKM description of the equivariant cohomology ring $H_T^*(\C P^1)$, hence
$$\cA_T(\C P^1) \simeq H_T^*(\C P^1)\; .$$

\begin{exm}\label{ex:cp1sq}
Let $M = \C P^1 \times \C P^1$ and let $T$ be a torus acting on $M$ by acting as in Example 2.1 on each copy of $\C P^1$,
$$\exp{\xi} \cdot \left( [z_0:z_1], [w_0:w_1] \right)= \left( [e^{2\pi i \alpha(\xi)} z_0:z_1] , [e^{2\pi i \alpha(\xi)} w_0:w_1] \right)\; .$$
There are five orbit strata: four correspond to the fixed points 
$$p_1 = ([1:0],[1:0]), \; p_2=([1:0],[0:1]), \; p_3 = ([0:1],[1:0]), \; p_4 = ([0:1],[0:1])\; ,$$
and the open dense stratum $X_0 = M \setminus \{p_1, p_2, p_3, p_4\}$.
The corresponding stabilizers are $\ft$ for the first four and $\ft_{\alpha} = \ker \alpha$ for the fifth. A polynomial assignment $A$ is therefore determined by the four polynomials 
$f_1, f_2, f_3, f_4$ associated to the fixed points, subject to the conditions
$$f_1 \equiv f_2 \equiv f_3  \equiv f_4 \pmod \alpha\; .$$
An example of a polynomial assignment on $M$ is $f_1=f_4=\alpha$, $f_2=f_3=0$, and $f_0=0$. 
\end{exm}

\begin{exm}\label{ex:cp1qb}
Let $M = \C P^1 \times \C P^1 \times \C P^1$ and let $T$ be a torus acting on $M$ by acting with the same weight $\alpha$ on each copy of $\C P^1$,
$$\exp{\xi} \cdot \left( Q_1, Q_2, Q_3 \right)= \left( \exp{\xi} \cdot Q_1 , \exp{\xi} \cdot  Q_2, \exp{\xi} \cdot Q_3 \right)\; .$$
There are nine orbit strata: eight correspond to the fixed points 
$p_1 = (S_1, S_2, S_3)$, $p_2 = (S_1, N_2, N_3)$, and so on 
\ldots (here $S_i$ and $N_i$ are the South and North poles of 
each of the three copies of $S^2 \simeq \C P^1$). The ninth 
stratum is the open dense stratum $X_0 = M \setminus M^T$. 
An example of a polynomial assignment on $M$ is the map that attaches 
$\alpha$ to $p_1$, $-\alpha$ to $p_2$ and zero to all other strata. 
\end{exm}

\section{Cohomological Assignments}
\label{sec:coh_assign}
A $T$-action on $M$ induces a $T$-action on $\Omega(M)$, and we denote by $\Omega(M)^T$ the space of $T$-invariant differential forms.
In the Cartan model 
$$(\Omega_T^*(M) = \Omega(M)^T \otimes \S(\ft^*), d_T)$$
of equivariant DeRham cohomology, a cochain $\phi \in \Omega_T^*(M)$  is a polynomial map 
$$\phi \colon \ft \to  \Omega(M)^T$$
and the coboundary operator on the space of such maps is given by
$$(d_T\phi) (\xi) = d(\phi(\xi)) + \iota(\xi_M) \phi(\xi)\,,$$
where $\iota(\xi_M)$ is the interior product with the  vector field $\xi_M$ associated to $\xi$.

The space $\Omega_T^*(M)$ has a grading given by
$$\Omega_T^{k}(M) = \bigoplus_{i+2j=k} \Omega^{i}(M)^T \otimes \S^{j}(\ft^*)$$
and an element $\phi \in \Omega_T^{2k}(M)$ is represented as
$$\phi = \phi_0 + \phi_2 + \dotsb + \phi_{2k}$$
where $\phi_{2i} \colon \ft \to \Omega^{2i}(M)^T$ is a polynomial map of degree $k-i$. 
Then $d_T \phi = 0$ implies
$$d(\phi_0(\xi)) + \iota(\xi_M) \phi_2(\xi) = 0\, ,$$
and hence $d(\phi_0(\xi)) = 0$ on the set where $\xi_M=0$. 

Let $X$ be an infinitesimal stratum of $M$ and $\ft_X \subset \ft$ the 
common infinitesimal stabilizer of all points in $X$. If $\xi \in \ft_X$, then $\xi_M=0$ on $X$, 
hence $d(\phi_0(\xi)) = 0$ on $X$. Therefore $\phi_0(\xi)$ is a constant function on $X$. 
The value of this constant is a polynomial of $\xi$, hence every $d_T$-closed 
cochain $\phi \in \Omega_T(M)$ defines a polynomial assignment 
$$X \to A(X) = \left. \phi_0\right|_{\ft_X} \in \S(\ft_X^*)\; .$$

If $\phi = d_T \psi$ is a $d_T$-exact cochain, then 
$\phi_0(\xi) = \iota_{\xi_M} (\psi_1(\xi))$. If $X$ is an infinitesimal 
stratum and $\xi \in \ft_X$, then $\xi_M = 0$ on $X$, so 
$\phi_0 = 0$ on $\ft_X$ and therefore $A(X) = 0$. Consequently we 
have a canonical grade-preserving ring 
homomorphism
$$\gamma \colon H_T^{even}(M) \to \cA_T(M)\; $$
from the $T$-equivariant cohomology of $M$ to the space of polynomial assignments,
\begin{equation}\label{eq:Phi}
\gamma([\phi])(X) = \left. \phi_0\right|_{\ft_X}\; ,
\end{equation}
for all $\phi = \phi_0 + \phi_2 + \dotsb + \phi_{2k} \in \Omega_T^{2k}(M)$ and 
$X \in \cP_T(M)$.
\begin{defn}
An assignment $A \in \cA_T(M)$ is called a \emph{\spass} if it is in the image of the 
homomorphism $\gamma$. The subring of \spass{s} is the image of $\gamma$ and is denoted by 
$\img$.
\end{defn}

The Cartan complex mentioned at the beginning of the section can be regarded as a 
double complex (see \cite[Section 6.5]{GS}), and the $E_1$ term
of the spectral sequence associated to it is 
$$
H(M)\otimes \S(\gd)\;.
$$

\begin{defn}
The $T$-manifold $M$ is \emph{equivariantly formal} if the spectral sequence associated to the
Cartan double complex collapses at the $E_1$ stage.
\end{defn}

As a consequence of this, if $M$ is equivariantly formal, then
\begin{equation}\label{eq:ief}
H_T^*(M)\simeq H(M)\otimes \S(\gd)\;.
\end{equation}
as  $\S(\gd)$-modules.
This implies the following
\begin{prop}\label{free}
If $M$ is equivariantly formal then $H^*_T(M)$ is a free $\S(\gd)$-module.
\end{prop}
\begin{proof}
This follows from observing that the right hand side of \eqref{eq:ief} is a free $\S(\gd)$-module.
\end{proof}

There are several conditions which imply equivariant formality (see for example \cite{GKM}).
In this paper we will mostly deal with the following situation.
Let $(M,\omega)$ be a compact symplectic manifold. Then Ginzburg \cite{Gi}
and Kirwan \cite{Ki} independently proved the following.
\begin{thm}\label{th:KG}
Let $T$ act on $(M,\omega)$ by symplectomorphisms, hence $\omega\in \Omega^2(M)^T$.
Then $M$ is equivariantly formal if $\omega$ admits a closed equivariant extension, \emph{i.e.} 
there exists $\Phi$ such that
$\omega-\Phi$ belongs to $ \Omega_T^2(M)$ and is closed under $d_T$.
\end{thm}
We will see in Section~\ref{sec:assign_hamiltonian} that $\Phi$ is 
called a \emph{moment map}, so this condition is
equivalent to requiring the action to be \emph{Hamiltonian}.

Equivariantly formal spaces are particularly important because
the restriction map
\begin{equation}\label{eq:i*}
i^*\colon H_T^*(M) \to H_T^*(M^T)
\end{equation}
has a very beautiful property.
We first recall the \emph{Abstract Localization Theorem} (see \cite{AB} and \cite{GS}).
\begin{thm}\label{th:absloc}
Let $M$ be a compact $T$-manifold. Then the kernel of the mapping \eqref{eq:i*} is the module of torsion
elements in $H^*_T(M)$. 
\end{thm}
As we observed before, if $M$ is equivariantly formal then $H_T^*(M)$ is a free $\S(\gd)$-module.
So Theorem \ref{th:absloc} has the following important consequence.
\begin{thm}\label{th:inj}
Let $M$ be a compact equivariantly formal $T$ space. Then
the restriction map \eqref{eq:i*} is injective.
\end{thm}
In particular we have that $M^T$ is non-empty.
By combining Theorem \ref{th:KG} with Theorem \ref{th:inj} we have the following result.
\begin{thm}\label{th:KIT}
If $(M,\omega)$ is a compact symplectic manifold with a Hamiltonian
$T$-action the map \eqref{eq:i*} is injective.
\end{thm}
This result is known as the Kirwan Injectivity Theorem and was proved by
Kirwan in \cite{Ki} using Morse-theoretic techniques.

\begin{prop}\label{even injective}
 If $M$ is equivariantly formal and $M^T$ is zero-dimensional, 
then the canonical map $\gamma \colon H_T^{even}(M) \to \cA_T(M)$ is injective.
\end{prop}

\begin{proof} Let $i_{M^T} \colon M^T \to M$ be the inclusion of the fixed point set $M^T$ into $M$ and
$$i_{M^T}^* \colon H_T^*(M) \to H_T^*(M^T) = \maps(M^T, \S(\ft^*))$$
the induced pull-back map. The connected components of $M^{\ft}=M^T$ are the 
fixed points, hence we have a natural map $r_T \colon \cA_T(M) \to \maps(M^T, \S(\ft^*))$, restricting a polynomial assignment 
to the fixed points.

If $\phi \in H_T^*(M)$, then, 
$$(i_{M^T}^*(\phi))(p)= \phi_0(p) = \gamma(\phi)(\{p\}) \; ,$$
showing $i_{M^T}^* = r_T \circ \gamma$,  hence that the following diagram commutes
$$
\begin{CD}
H_T^{even}(M) @>\gamma >> \cA_T(M) \\
@Vi_{M^T}^*VV @VV r_T V \\
H_T^*(M^T) @= \maps(M^T, \S(\ft^*))
\end{CD} \; .
$$
Since $M$ is equivariantly formal, the Kirwan Injectivity Theorem implies that $i_{M^T}^*$ is injective, and therefore the canonical map $\gamma$ is injective as well.
\end{proof}

Note that, under the above assumptions, the odd cohomology vanishes.

\begin{prop}\label{odd cohom}
If $M$ is equivariantly formal and $M^T$ is zero-dimensional, 
then $H_T^{odd}(M)=0$.
\end{prop}
\begin{proof}
Since $M^T$ is zero dimensional, $H_T^{odd}(M^T)=0$. The result follows from the fact that
the grade preserving map
$i_{M^T}^*$ is injective.
\end{proof}
By combining Propositions \ref{even injective} and \ref{odd cohom} we obtain the following result.
\begin{thm}\label{isom ef}
If $M$ is equivariantly formal and $M^T$ is zero-dimensional, then the equivariant cohomology ring $H_T^*(M)$ is isomorphic to the
ring of cohomological assignments $\mathcal{H}_T(M)$.
\end{thm}

\begin{defn}
The \emph{\quotmod{} module} is the quotient module $\cA_T(M)/\img$.
\end{defn}

In general the \quotmod{} module is not trivial, as shown by the next example. (See also the examples in \cite[Section 4]{GH}.) 

\begin{exm}\label{ex:not-cohomological}
The assignment $f$ in Example~\ref{ex:cp1sq} is not a \spass: if $f = \gamma(\phi)$, 
then by the Atiyah-Bott Berline-Vergne localization formula,
$$\int_M \phi = \frac{f_1}{\alpha^2} - \frac{f_2}{\alpha^2} - \frac{f_3}{\alpha^2} + \frac{f_4}{\alpha^2} = \frac{2}{\alpha}\; .$$
But the integral must be a polynomial in $\S(\ft^*)$, and $2/\alpha$ obviously isn't. This shows that the polynomial assignment $f$ is not in the image of $\gamma$, hence $\gamma$ is not surjective.
\end{exm}

As we have already observed, equivariant formality implies that $M^T\neq \emptyset$. Moreover we have the following.
\begin{lem}\label{minimal strata}
Let $M$ be an equivariantly formal $T$-space. Then the minimal strata, i.e.\ the strata which are
minimal under the $\preccurlyeq$ ordering, coincide with the fixed point set $M^T$.
\end{lem}
\begin{proof}
It is sufficient to prove that each stratum $X$ contains a fixed point. Let $\mathfrak{t}_X$ be the common infinitesimal stabilizer
of the points in $X$ and $H=\mathrm{exp}(\mathfrak{t}_X)\subset T$. 
By Proposition \ref{free} $H_T(M)$ is a free $\mathbb{S}(\mathfrak{t}^*)$-module, and
the conclusion follows from \cite[Theorem 11.6.1]{GS}.
\end{proof}

In the rest of this paper we study the image of the canonical map $\gamma$: 
we give conditions that an assignment must satisfy in order to be a cohomological 
assignment, and cases when all assignments are cohomological. We show that, 
under suitable conditions, the \quotmod{} module is a torsion module over $\S(\ft^*)$ 
(see Corollary \ref{defect torsion}).

\begin{exm}[{\em Delta classes}]
\label{delta}
Let $M$ be an equivariantly formal $T$-space with isolated fixed points and $p\in M^T$.
Since $p$ is the only fixed point of the isotropy action of $T$ on $T_pM$, this vector
space can be endowed with a complex structure with respect to which
the weights of the complex representation are given by $\alpha_1,\ldots,\alpha_n\in \mathfrak{t}^*$.
Note that this complex structure is not unique, hence the signs of the weights are not well-defined.
However the content of this example will not be affected by this ambiguity.
Let $\{\beta_1,\ldots,\beta_k\}\subset \mathfrak{t}^*$ be a choice of pair-wise independent vectors such that 
for every $i=1,\ldots,n$, there exist $j\in \{1,\ldots,k\}$ and $\lambda_i\in \R\setminus \{0\}$ with $\alpha_i=\lambda_i\beta_j$.
Define the \emph{delta class at $p$} to be the map $\delta_p\colon M^T\to \mathbb{S}(\mathfrak{t}^*)$, where
$$
\delta_p(q)=
\begin{cases}
\displaystyle\prod_{j=1}^k\beta_j& \mbox{if}\;\;p=q\\
& \\
0 & \mbox{otherwise}
\end{cases}\;.
$$
\begin{prop}\label{delta class}
For each $p\in M^T$ the delta class $\delta_p$ defines a polynomial assignment. 
\end{prop}
\begin{proof}
Observe that, by Lemma \ref{minimal strata}, $M^T$ coincides with the minimal strata. In order to see that
$\delta_p$ defines a polynomial assignment, it is sufficient to prove that for each stratum $X\in\mathcal{P}_T^1(M)$
and each $q_1,q_2\in \overline{X}\cap M^T$ we have $\delta_p(q_1)|_{\mathfrak{t}_X}=\delta_p(q_2)|_{\mathfrak{t}_X}=0$.
If $p\notin \overline{X}\cap M^T$ or if $p\neq q_i$ for $i=1,2$,  this follows from the definition of $\delta_p$.
So suppose that $p\in \overline{X}\cap M^T$ and $p=q_1$. 
Since the exponential map $\mathrm{exp}\colon T_pM\to M$ intertwines the $T$-action on $T_pM$ and $M$,
it's easy to see that there exists $j\in\{1,\ldots,k\}$
such that $\mathfrak{t}_X=\mathrm{ker}(\beta_j)$, hence $(\prod_{j=1}^k\beta_j)|_{\mathfrak{t}_X}=0$
and the conclusion follows. 
 
\end{proof}
\end{exm}

\section{Assignments and GKM Spaces}
\label{sec:GKM_assign}

When $M$ is an equivariantly formal compact manifold, the minimal strata are the connected components of $M^T$. Therefore an assignment is determined by fixed point data, subject to compatibility conditions, hence the restriction
$$r_T \colon \cA_T(M) \to \maps(M^T, \S(\ft^*))$$
is injective. If the fixed points are isolated, then $\gamma$ is an isomorphism if and only if $i_{M^T}^*$ and $r_T$ have the same image in $\maps(M^T, \S(\ft^*))$. Since $\im(i_{M^T}^*) \subseteq \im(r_T)$, that means that the canonical map $\gamma$ is an isomorphism if and only if $\im(r_T) \subseteq \im(i_{M^T}^*)$, or, in other words, if and only if the compatibility conditions imposed on fixed point data are sufficient to guarantee that the polynomial assignment comes from an equivariant class.

Let $M$ be
a compact equivariantly formal $T$-manifold, 
$K_1,\ldots, K_N$ the family of codimension one subtori of $T$ which occur as
isotropy groups of $M$, and $M^{K_j}$ the subset of points of
$M$ fixed by $K_j$, for every $j=1,\ldots,N$. Since $M, M^T$ and the $M^{K_j}$s are $T$-invariant spaces, 
the commutative diagram given by $T$-invariant inclusions 
$$\xymatrix{M^T \ar[rr]^{i_{M^T}} \ar[dr]_{i_{K_j}} & & M \\
& M^{K_j} \ar[ur] & }$$
induces a commutative diagram in equivariant cohomology
$$\xymatrix{H_T^*(M) \ar[rr]^{i_{M^T}^*} \ar[dr] & & H_T^*(M^T) \\
& H_T^*(M^{K_j}) \ar[ur]_{i_{K_j}^*} & }.$$
Hence
 \begin{equation*}
i_{M^T}^*(H_T^*(M))\subseteq \bigcap_{j=1}^N i^*_{K_j}(H_T^*(M^{K_j}))\;.
\end{equation*}
The opposite inclusion is a well-known theorem by Chang and Skjelbred (cf.
\cite[Lemma 2.3]{CS}). 
\begin{thm}[Chang-Skjelbred]\label{th:CS_theorem}
Let $M$ be a compact and equivariantly formal $T$-manifold. 
Let $K_1\ldots,K_N$ be the subtori of $T$ of codimension one which occur as isotropy groups
of points of $M$. Then
\begin{equation}\label{eq:CS=}
i_{M^T}^*(H_T^*(M))= \bigcap_{j=1}^N i^*_{K_j}(H_T^*(M^{K_j}))\;.
\end{equation}
\end{thm}

\begin{remark}
By Proposition~\ref{free}, equivariant formality implies that $H_T^*(M)$ is a free $\S(\ft^*)$-module. As pointed out in \cite{AFP}, the Chang-Skjelbred Theorem is valid when the equivariant formality condition is replaced by the milder condition that $H_T^*(M)$ be a reflexive $\S(\ft^*)$-module.
\end{remark}

The spaces $M^{K}$ are acted effectively on only by the quotient 
circle $S^1\simeq T/K$; thus Theorem \ref{th:CS_theorem} asserts that in 
order to study $\im(i_{M^T}^*)$ it is necessary and sufficient to understand 
the $S^1$-equivariant cohomology of $M^{K}$. Indeed, if $\mathfrak{k}$ denotes 
the Lie algebra of $K$, we have
\begin{equation}\label{eq:equiv_cohom}
H_T^*(M^K)=H^*_{T/K}(M^K)\otimes \mathbb{S}(\mathfrak{k}^*)\;.
\end{equation}

This is particularly useful when the submanifolds $M^{K}$ are simple: 
that is the case when $M$ is a GKM space. Recall that a GKM space is a compact, equivariantly formal (oriented) $T$-manifold $M$, with isolated fixed points and such that the orbit type strata with stabilizer of codimension one are two-dimensional. If $M$ is such a space and $K$ is a codimension one subtorus of $T$, then the connected components of $M^K$ are either isolated fixed points or two-spheres that contain exactly two fixed points in $M^T$. 
It's easy to see that this condition is equivalent to requiring 
the weights of the isotropy action of $T$ on $T_pM$
be pair-wise linearly independent for each $p\in M^T$.
\begin{thm} \label{characterization}
Let $M$ be a $T$-equivariantly formal space with isolated fixed points $M^T$. Then
$$
H_T^*(M)\simeq \mathcal{A}_T(M)\;\mbox{ if and only if } M\mbox{ is GKM.}
$$
\end{thm}
\begin{proof} {\em(If)}
We first prove that if $M$ is GKM then $H_T^*(M)\simeq \mathcal{A}_T(M)$.
By Theorem \ref{isom ef} we know that $H_T^*(M)\simeq \mathcal{H}_T(M)=\gamma(H_T^{even}(M))$.
Therefore we need to prove that the canonical map $\gamma\colon H_T^{even}(M)\to \mathcal{A}_T(M)$
is an isomorphism, and it suffices to show that
$\im(r_T) \subseteq \im(i_{M^T}^*)$. 

Let $f=r_T(\phi)\colon M^T \to \S(\ft^*)$; it suffices to show that $f \in i_{M^T}^*$, or, equivalently, that $f$ is the restriction to $M^T$ of an equivariant cohomology class on $M$. 

By the Chang-Skjelbred Theorem,
$$\im(i_{M^T}^*) = \bigcap_{K} i_K^* H_T^*(M^K)\; ,$$
where $i_K \colon M^K \to M^T$ is the inclusion map and the intersection is over  codimension one subtori $K \subset T$. Since $M$ is a GKM space, 
the connected components of such an $M^K$ are either fixed points or two spheres 
on which $T$ acts as in Section~\ref{sec:assignments}. 

Let $K$ be a codimension one subtorus of $T$, and let $\fk$ be the Lie algebra 
of $K$. Let $X$ be an infinitesimal orbit type stratus such that $\ft_X = \fk$. 
Then the closure of $X$ is a two-sphere on which $T$ acts by rotations. By the 
result in Section \ref{sec:assignments}
$$\cA_T(\bar{X}) = H_T^*(\bar{X})\; ,$$
and therefore 
$$\left. f\right|_{\bar{X}^T} \in i_K^* H_T^*(\bar{X}) \; .$$
Therefore
$$f \in  \bigcap_{K} i_K^* H_T^*(M^K) = \im(i_{M^T}^*) \; ,$$
hence $\im(r_T) \subseteq \im(i_{M^T}^*)$ and $\gamma$ is an isomorphism.\\

{\em (Only if)} Now we prove the converse, i.e.\ we prove that if $M$ is not GKM then $H_T^*(M)$ is not isomorphic to $\mathcal{A}_T(M)$.
Hence suppose that there exists a connected component $N$ of $M^K$ of dimension strictly greater than $2$,
where $K\subset T$ is a codimension one subtorus. Then, by \cite[Theorem 11.6.1]{GS},
$N$ contains a fixed point $p$. Let $\{\alpha_1,\ldots,\alpha_n\}$ be the weights of the isotropy representation
of $T$ on $T_pM$ and $\{\alpha_1,\ldots,\alpha_l\}$ the ones on $T_pN$.  
Since $\mathrm{exp}\colon T_pM\to M$ intertwines the $T$ actions, we have
$\alpha_i|_{\mathfrak{t}_N}=0$ for $i=1,\ldots,l$, and
since $\mathfrak{t}_N=Lie(K)$ is of codimension one in $\mathfrak{t}$ we have
$\alpha_i$ is proportional to $\alpha_j$ for every $i,j\in \{1,\ldots,l\}$.
Let $\delta_p$ be the delta class introduced in section \ref{delta}.
By the previous argument, the number of pair-wise independent vectors $\beta_i$s
defined in section \ref{delta} is strictly less than $n$, the total number of $\alpha_i$s. Hence 
\begin{equation}\label{delta not c}
\sum_{q\in M^T}\frac{\delta_p(q)}{e_M(q)}=\frac{\prod_{i=1}^k\beta_i}{\prod_{i=1}^n\alpha_i}\notin\mathbb{S}(\mathfrak{t}^*).
\end{equation}
So $\delta_p$ cannot be a cohomological assignment, since if it were, by the Atiyah-Bott-Berline-Vergne localization formula the
sum in \eqref{delta not c} would be in $\mathbb{S}(\mathfrak{t}^*)$.
\end{proof}

Therefore, we have that for
$T$-equivariantly formal spaces with isolated fixed points, the \quotmod{} module is trivial
if and only if $M$ is GKM.
By an argument similar to the proof of the {\em (Only if)} in Theorem \ref{characterization} we also have that
\begin{prop}
The delta classes $\delta_p$ defined in section \ref{delta} are cohomological if and only if $M$ is GKM.
\end{prop}

\section{Assignments and Hamiltonian Spaces}
\label{sec:assign_hamiltonian}

A Hamiltonian $T$-space $M$ is a symplectic manifold $(M, \omega)$ with a $T$-action that leaves the symplectic form $\omega$ 
invariant, \emph{i.e.} $\omega\in \Omega^2(M)^T$, and for which there exists 
a moment map:
a $T$-equivariant polynomial map of degree one $\Phi \colon \ft \to \Omega^0(M)$, satisfying the condition
$$d \Phi(\xi)= \iota(\xi_M)\omega$$
for all $\xi\in \ft$, where $\xi_M$ is the vector field on $M$ associated 
to $\xi\in \mathfrak{t}$. The existence of a moment map $\Phi$ is equivalent to 
the existence of a $d_T$ closed equivariant extension $\omega-\Phi\in \Omega_T^2(M)$ 
of $\omega$: this is often referred to as the \emph{equivariant symplectic form}.

Assume that the fixed point set $M^T$ is discrete. Fix a generic $\xi$ in $\mathfrak{t}$, \emph{i.e.}
$\langle\alpha, \xi \rangle\neq 0$ for each weight $\alpha\in \mathfrak{t}^*$ in the isotropy representations of $T$ on the tangent
space $T_pM$, for each fixed point $p\in M^T$.
It is well known that $\varphi=\Phi(\xi)\colon M\to \R$ is a perfect $T$-invariant Morse function, whose critical set
coincides with $M^T$. Moreover the Morse index at a fixed point $p$ is $2\lambda_p$, where $\lambda_p$
denotes the number of weights $\alpha$ in the isotropy representation on $T_pM$ such that $\langle \alpha , \xi \rangle <0$.
Let $\Lambda_p^-$ be the product of such weights, and $\Lambda_p$ the product of all the weights of the isotropy representation 
on $T_pM$, i.e. $\Lambda_p$ is the Euler class of $M$ at $p$, which we also denote by $e_M(p)$. 

In \cite{Ki}, Kirwan used the Morse stratification associated to $\varphi$ to prove that $i^*\colon H_T(M)\to H_T(M^T)$ is
\emph{injective} (cf. Theorem \ref{th:KIT}) and the restriction map to ordinary cohomology
$$r\colon H_T(M)\to H(M)$$ 
is \emph{surjective}. Moreover she gave a natural basis for $H^*_T(M)$.

\begin{lem}[Kirwan]\label{lem:Kc}

For every $p\in M^T$, there exists $\gamma_p\in H_T^{2\lambda_p}(M)$ such that
\begin{itemize}
\item[(i)] $\gamma_p(p)=\Lambda_p^-$;
\item[(ii)] $\gamma_p(q)=0$ for all $q\in M^T\setminus \{p\}$ such that $\varphi(q)\leqslant \varphi(p)$.
\end{itemize}
(Here $\gamma_p(q)$ denotes the restriction of $\gamma_p$ to the fixed 
point $q\in M^T$.) Moreover the classes $\{\gamma_p\}_{p\in M^T}$ form a 
basis of $H_T^*(M)$ as an $\S(\gd)$-module, and
$\{r(\gamma_p)\}_{p\in M^T}$ form a basis of $H^*(M)$ as an $\R$-module.
\end{lem}

We refer to these cohomology classes as the \emph{Kirwan classes}.

\begin{remark}
Given the triple $(M,\omega,\varphi)$, the set of classes $\{\gamma_p\}_{p\in M^T}$ 
is not uniquely characterized by (i) and (ii). 
In Section~\ref{sec:cas} 
we  introduce a different set of classes which always exist whenever $T=S^1$ 
and are canonically associated to $(M,\omega,\varphi)$.
\end{remark}

As an immediate consequence of Lemma \ref{lem:Kc}, for every $k$ we have 
\begin{equation}\label{eq:dimensions}
\dim_{\R}(H^{2k}(M))=\#\{s\in M^T\mid \lambda_s=k\}\quad\mbox{and}\quad H^{2k+1}(M)=0
\end{equation}

\begin{lem}\label{lem:bigdim}
Let $S^1$ be a circle group and $(M,\omega)$ a compact Hamiltonian $S^1$-space of dimension $2d$, 
with discrete fixed point set $M^{S^1}$. If $k \geqslant d$, then the restriction
\begin{equation}\label{eq:mapbij}
i^* \colon H_{S^1}^{2k}(M)\to H_{S^1}^{2k}(M^{S^1})
\end{equation}
is bijective.
\end{lem}

\begin{proof}
By equivariant formality 
$$
H_{S^1}^{2k}(M)=\bigoplus_{i=0}^k\left(H^{2i}(M)\otimes \mathbb{S}^{k-i}(\R)\right)\,,
$$
hence if $k \geqslant d$, 
\begin{align*}\dim_{\R} H_{S^1}^{2k}(M) = & \sum_{i=0}^d \dim_{\R} H^{2i}(M) \cdot \dim_{\R} \S^{k-i}(\R) = 
\sum_{i=0}^d \dim_{\R} H^{2i}(M) =\\
= & \#M^{S^1} =\# M^{S^1} \cdot \dim_{\R} \S^k(\R) =\dim_{\R} H_{S^1}^{2k}(M^{S^1})  \; .
\end{align*}

Since $i^*$ is injective (by Kirwan), the equality of dimensions implies that $i^*$ is bijective.
\end{proof}

\section{Examples}
\label{sec:examples}

In this section we will describe from the assignment perspective some well-known examples of equivariant cohomology classes. For simplicity we will assume, as in the previous section, that the $T$-action is Hamiltonian and that $M^T$ is finite.

\begin{exm}[\emph{Equivariant Chern classes}] For every $p \in M^T$ let $\alpha_{i,p}$, $i=1, \ldots, n$ be the weights of the isotropy representation of $T$ on $T_pM$. We will denote by $\sigma_m(\alpha_{1,p}, \ldots, \alpha_{n,p})$ the $m$th elementary symmetric polynomial in the $\alpha_{i,p}'$s, \emph{i.e.}
\begin{equation}
\prod_{i=1}^n (1+\alpha_{i,p}x) = \sum_{m=0}^n \sigma_{m}(\alpha_{1,p}, \ldots, \alpha_{n,p}) x^m
\end{equation}
and by $c_m \colon M^T\to \S(\ft^*)$ the map
\begin{equation}\label{eq:vg2}
M^T \ni p \to \sigma_m(\alpha_{1,p}, \ldots, \alpha_{n,p})\; .
\end{equation}

We claim that this map defines an assignment. To see this, let $Z^\circ$ be an orbitype stratum, $K$ its stabilizer group, and $Z$ its closure. For $p \in Z^T$ let $\alpha_{1,p}, \ldots, \alpha_{d,p}$ be the weights of the isotropy representation of $T$ on $T_pZ$ and $\alpha_{d+1, p}, \ldots, \alpha_{n,p}$ the remaining weights. Since the isotropy representation of $K$ on $T_pZ$ is trivial, $\alpha_{i,p}|_\fk = 0$ for $i=1, \ldots, d$ and since $Z$ is connected the isotropy representation of $K$ on $T_pM$ is the same for all $p \in Z$ and hence the same for all $p \in Z^T$. Thus for the remaining weights the restrictions 
$$\alpha_{d+1,p}|_\fk, \ldots, \alpha_{n,p}|_\fk$$
are the same for all $p$, up to permutations. Hence the restriction 
$$\sigma_m(\alpha_{1,p}, \ldots, \alpha_{n,p}) |_\fk$$
is the same for all $p \in Z^T$, and hence the map  of $\mathcal{P}_T^r(M)$ into $\S(\ft^*)$ sending $Z$ to 
\begin{equation}
c_m(Z) = \sigma_m(\alpha_{1,p}, \ldots, \alpha_{n,p})|_{\fk}\quad, \quad p \in Z^T
\end{equation}
defines an assignment. One can show (see \cite[p.212]{GGK}) that these assignments are cohomological  assignments associated with the images in $\cA_{T}(M)$ of the equivariant Chern classes of $M$.
\end{exm}

\begin{exm}[\emph{Equivariant Thom classes}]\label{ex:vg1} Let $X$ be a compact 
$T$-invariant symplectic submanifold of $M$ and for each $p \in X^T$ let 
$\alpha_{i,p}^N$, $i=1, \ldots, s$ be the weights of the representation of $T$ on 
the normal space to $X$ at $p$. We will show that
\begin{equation}\label{eq:vg4}
A_X(p) =  
\begin{cases}
0, & \text{ if } p \not\in X^T \\
\alpha_{1,p}^N \alpha_{2,p}^N \cdots \alpha_{s,p}^N, & \text{ if } p \in X^T
\end{cases}
\end{equation}
defines an assignment on $M$. To see this let $Z^{o}$ be an orbitype stratum, $Z$ its closure and $\fk = \fk_Z$ the infinitesimal stabilizer of $Z^{o}$. Then if $p \in Z^T$ is not in $X$,
$$A_X(p) |_{\fk} = 0$$
and if $p$ is in $X$, one of the weights $\alpha_{i,p}^N$ has to be a weight of the representation of $T$ on $T_pZ$ and hence $\alpha_{i,p}^N |_{\fk} = 0$. Hence $A_X(p) |_{\fk} = 0$. If $Z^T$ and $X$ are disjoint or if $T_pZ \cap N_pX$ is non-zero for all $p \in Z^T \cap X$, then $A_X(p)|_{\fk} = 0$ for all $p \in Z^T$.

Suppose now that for some $p \in Z^T \cap X$, $T_pZ$ is contained in $T_pX$. To handle this case we'll prove the following result.

\begin{thm}
If $T_pZ \subseteq T_pX$ for some $p \in X \cap Z$, then $Z \subseteq X$.
\end{thm}

\begin{proof}
Let $\exp_p \colon T_pX \to X$ be a $T$-invariant exponential map. Then $\exp_p$ maps an open neighbourhood $\mathcal{U}$ of 0 in $T_pX$ diffeomorphically onto a neighborhood $V$ of $p$ in $X$ and maps $\mathcal{U} \cap T_pZ$ diffeomorphically onto $V \cap Z$. Thus if $Y$ is a connected component of $X \cap Z$, the set of $p'$s in $Y$ for which $T_pZ$ is contained in $T_pX$ and the set for which it's not are both open. By connectivity it follows that if $T_pZ$ is contained in $T_pX$ for some $p \in Y$, then $Y =Z$.
\end{proof}

As a corollary of this it follows that if $T_pZ$ is contained in $T_pX$ for some $p \in Z^T$ (and hence for all $p \in Z^T$), then the normal weights of the representation of $T$ on $N_pZ$ contain the normal weights, $\alpha_{i,p}^N$, of the representation of $T$ on $N_pX$ and hence as in the previous example, the weights $\alpha_{i,p}^N |_{\fk}$ are, up to permutations, the same for all $p \in Z^T$. In other words $A_X(p)|_{\fk}$ is the same for all $p \in Z^T$. This proves that the map $A_X$ extends to an assignment. (As in the previous example this is a cohomological assignment: the image under the map $\gamma$ of the Thom class of $X$. For details, see \cite[Ch. 10]{GS}).
\end{exm}

We'll conclude this section by describing an integral formula for assignments in which the Chern assignments that we defined above play a basic role. Let $[\eta]$ be an element of $H_T^{2m}(M)$, $\gamma([\eta])$ its associated assignment and $f \colon M^T \to \S(\ft^*)$ its restriction to $M^T$. Then if $Z$ is an infinitesimal orbit type stratum of dimension $2k$ and $Ch_n$ the top dimensional Chern form of $M$, one gets from \eqref{eq:abbv}
\begin{equation*}
\int_{\bar{Z}} Ch_n \eta = \sum_{p \in (\bar{Z})^T} \frac{Ch_n(p) f(p)}{\prod_{j=1}^k \alpha_{j,p}}  = \sum \prod_{j=k+1}^n \alpha_{j,p} f(p)\; , 
\end{equation*}
the expression on the left being an element of $\S^{m+\ell}(\ft^*)$, with $\ell= n-k$. If one restricts this polynomial to $\ft_Z$, each of the summands is equal to $c_\ell \gamma([\eta])(Z)$ and since the number of $p$'s is just the Euler characteristic $\chi_Z$ of $\bar{Z}$, we obtain
\begin{equation}
\left. \left( \int_{\bar{Z}} c_n \eta \right) \right|_{\ft_Z} = \chi_Z (c_{\ell}\gamma([\eta]))(Z) \; .
\end{equation}
Hence from the integral on the left one can recover the value that the assignment $\gamma([\eta])$ assigns to each orbit-type stratum.

\section{A Characterization of Cohomological Assignments}
\label{sec:cas}

Example~\ref{ex:not-cohomological} shows that cohomological assignments must satisfy
natural conditions imposed by the localization formula. 

\begin{prop}\label{prop:nec}
A necessary condition for an assignment $f$ to be a \spass{} is that  for every 
closure $Y = \overline{X}$ of an infinitesimal stratum $X$, and for every 
equivariant class $\eta \in H_T^*(Y)$, 
\begin{equation}\label{eq:polypair}
\sum_{p \in Y^T} \frac{f(p)\eta(p)}{e_Y(p)} \, \in \, \S(\ft^*)\; ,
\end{equation}
where $e_Y(p)$ is the equivariant Euler class of $Y$ at $p$. 
\end{prop}

When $\overline{X}$ is of dimension at most four for all strata $X$ of degree one, 
by a result of Goldin and Holm (\cite[Theorem 2]{GH})
it suffices to check \eqref{eq:polypair} only for $\eta=1$.

\begin{thm}[Goldin-Holm]\label{thm:GH}
Let $M$ be a compact, connected, Hamiltonian $T$-space, 
such that the $T$-action is effective and has isolated fixed points.
Moreover, suppose that for every infinitesimal stratum $X$ of degree one, the closure 
$Y=\overline{X}$ has dimension at most four. Then an assignment $f$ is cohomological 
if and only if 
\begin{equation}\label{eq:integral}
\sum_{p \in Y^T} \frac{f(p)}{e_Y(p)} \, \in \, \S(\ft^*)\; ,
\end{equation}
for every infinitesimal stratum $X$ of degree one, 
where $e_Y(p)$ is the equivariant Euler class of $Y=\overline{X}$ at $p$.
\end{thm}

If some of the closures have dimension higher than four, 
that is no longer the case. The following is a counter-example.

\begin{exm}
The assignment $f$ in Example~\ref{ex:cp1qb} satisfies \eqref{eq:integral}, as
$$\sum_{p \in M^T} \frac{f(p)}{e_M(p)} = \frac{\alpha}{\alpha^3} + \frac{-\alpha}{\alpha^3} = 0\; .$$
Nevertheless, $f$ is not a \spass. If it were, then so would be $f^2$; but
$$\sum_{p \in M^T} \frac{f^2(p)}{e_M(p)} = \frac{\alpha^2}{\alpha^3} + \frac{(-\alpha)^2}{\alpha^3} = \frac{2}{\alpha}\; ,$$
which shows that  $f^2$ is not in the image of $\gamma$, hence neither is $f$.
\end{exm}

The next result generalizes the example above.

\begin{prop}
Let $(M,\omega)$ be a compact, connected, Hamiltonian $T$-space of dimension $2n\geqslant 6$,  with isolated fixed points 
such that for every codimension one subtorus $K \subset T$, $M^K = M^T$ or $M^K = M$. 
Then there exist assignments that satisfy \eqref{eq:integral} but are not cohomological.
\end{prop}

\begin{proof}
The conditions of the hypothesis imply that all the isotropy weights are collinear; let $\alpha$ be their direction. If $K = \exp{\ker \alpha}$, then $M^K = M$, otherwise $M^K = M^T$. Let 
$$M^T = \{ p_1, \ldots, p_N\}$$
be the set of fixed points. Since $[\omega]^k\neq 0$ for all $0\leqslant k\leqslant \dim(M)/2$, from \eqref{eq:dimensions} it follows that $N = \dim_{\R} H^*(M) \geqslant
n+1\geqslant 4$ and
$$H_T^*(M) \simeq H^*(M) \otimes \S(\ft^*)$$
is a free $\S(\ft^*)$-module of rank $N$.
For each fixed point $p_k$ we have 
$e_M(p_k) = \lambda_k \alpha^n$, for a non-zero real  $\lambda_k$. 
There exist two values $\lambda_i$ and $\lambda_j$ that do not add up to zero, otherwise, 
all the $\lambda$'s would be zero. Without loss of generality, we can assume that 
$\lambda_1+\lambda_2 \neq 0$.
Let $f \colon \cP_T(M) \to \S(\ft^*)$ be the assignment determined by 
$f(p_1) = \lambda_1 \alpha$, $f(p_2) = -\lambda_2 \alpha$, and $f(p_k) = 0$ for $k \geqslant 3$.

Then 
$$\sum_{k=1}^N \frac{f(p_k)}{\lambda_k \alpha^n} = \frac{\lambda_1 \alpha}{\lambda_1 \alpha^n} + \frac{-\lambda_2\alpha}{\lambda_2\alpha^n}=  0\; ,$$
but $f$ is not cohomological: if it were, then so it would be $f^2$, and that is impossible, because
$$\sum_{k=1}^N \frac{f^2(p_k)}{\lambda_k \alpha^n}  = \frac{\lambda_1^2 \alpha^2}{\lambda_1 \alpha^n} + \frac{\lambda_2^2 \alpha^2}{\lambda_2 \alpha^n} = \frac{\lambda_1 + \lambda_2}{\alpha^{n-2}}\; ,$$
is not a polynomial in $\S(\ft^*)$, since by assumption $n\geqslant 3$. 
\end{proof}

In the following we discuss necessary and sufficient conditions for an assignment to be 
in the image of the canonical map $\gamma$, hence to be cohomological.

Let $M$ be a compact Hamiltonian $T$-manifold, 
with moment map $\Phi\colon \ft\to \Omega^0(M)$, 
and suppose that the fixed point set $M^T$ is discrete. Fix a generic $\xi$ in $\mathfrak{t}$, 
and let $\varphi=\Phi(\xi)\colon M\to \R$ (see section 5).

\begin{defn}
An equivariant cohomology class $\tau_p$ is called a \emph{canonical class} at $p\in M^T$ if $\tau_p\in H_T^{2\lambda_p}(M)$ and
\begin{itemize}
\item[(i)'] $\tau_p(p)=\Lambda_p^-$ ;
\item[(ii)'] $\tau_p(q)=0$ for all $q\in M^T\setminus \{p\}$ such that $\lambda_q\leqslant \lambda_p$.
\end{itemize}
(Here $\tau_p(q)$ denotes the restriction of $\tau_p$ to the fixed point $q\in M^T$.)
\end{defn}

Canonical classes were introduced by Goldin and Tolman \cite{GT}, 
inspired by previous work of Guillemin and Zara \cite{GZ} (see also \cite{ST}). 
They do not always exist, however when they exist they are uniquely characterized by 
$(M,\omega,\varphi)$ and satisfy property (ii) of Lemma \ref{lem:Kc}. 
So canonical classes are a particular choice of Kirwan classes. Hence,
if there exists a canonical class $\tau_p$ for every $p\in M^T$, 
then $\{\tau_p\}_{p\in M^T}$ is a basis
for $H_T^*(M)$ as an $\mathbb{S}(\mathfrak{t}^*)$-module 
(see also Proposition 2.3, Lemma 2.7 and Lemma 2.8 in \cite{GT}).
 
Let $T=S^1$ and $Lie(S^1)=\fs$. 
Then canonical classes always exist (see \cite[Lemma 1.13]{MT}). 
Moreover, in this case they have a few special properties:
\begin{enumerate}
\item The restrictions of $\{ \tau_p\}_{p\in M^{S^1}}$ to $M^{S^1}$ form a basis of $\text{Maps}(M^{S^1}, Q(\fs^*))$ over $Q(\fs^*)$, the field of fractions of $\S(\fs^*)$.
\item There exists a similar basis $\{ \tau_p^{*}\}_{p\in M^{S^1}}$ (obtained by reversing the orientation of $S^1$), dual to $\{ \tau_p\}_{p\in M^{S^1}}$ in the sense that
$$\sum_{r \in M^{S^1}} \frac{\tau_p(r) \tau_q^{*}(r)}{e_M(r)} = \delta_{p,q}\; .$$
\end{enumerate}

A consequence of these properties is that for such spaces, the condition given in Proposition~\ref{prop:nec} is not only necessary, but also sufficient.

\begin{thm}\label{milena}
Let $M$ be a compact Hamiltonian $S^1$-manifold with discrete 
fixed point set $M^{S^1}$, and $\{\tau_p\}_{p\in M^{S^1}}$ the canonical classes for a fixed 
orientation of $\fs^*$. A polynomial assignment $f \in \Aktm{}{S^1}{M}$ is a \spass{} if 
and only if for all $q \in M^{S^1}$,
\begin{equation}\label{mc}
\sum_{p \in M^{S^1}} \frac{f(p)\tau_q(p)}{e_M(p)} \in \S(\fs^*)\; .
\end{equation}
\end{thm}
\begin{proof}
If $f$ is a cohomological assignment then \eqref{mc} is a consequence of the Atiyah-Bott-Berline-Vergne Localization
theorem. We need to prove the converse.
By (1) above, there exists $\{\alpha_p\}_{p\in M^{S^1}}\subset Q(\fs^*)$ such that $f=\sum_{p\in M^{S^1}}\alpha_p\tau_p^*$. 
Hence it is sufficient to prove that  $\{\alpha_p\}_{p\in M^{S^1}}\subset \S(\fs^*)$.
Observe that for every $q\in M^{S^1}$
\begin{equation*}
\begin{aligned}
\sum_{r\in M^{S^1}} 
\frac{f(r)\tau_q(r)}{e_M(r)}= & \sum_{r\in M^{S^1}}\frac{(\sum_{p\in M^{S^1}}\alpha_p\tau_p^*(r))\tau_q(r)}{e_M(r)}= \\  = &  \sum_{p\in M^{S^1}}
\alpha_p\frac{\sum_{r\in M^{S^1}}\tau_p^*(r)\tau_q(r)}{e_M(r)}  = \sum_{p\in M^{S^1}}\alpha_p\delta_{p,q}=\alpha_q \;,\\
\end{aligned}
\end{equation*}
where we use property (2) above for the second-last equation. The conclusion follows from \eqref{mc}.
\end{proof}

The result in Theorem \ref{milena} agrees with \cite[Theorem 1.6]{P}; however the proof is different.
The following Corollary is an easy consequence of Theorem \ref{milena}.
\begin{corollary}\label{defect torsion}
Let $M$ be a compact Hamiltonian $S^1$-manifold with discrete fixed point set $M^{S^1}$. Then the
\quotmod{} module is a torsion module over $\mathbb{S}(\fs^*)=\R[x]$.
\end{corollary}
\begin{proof}
We need to prove that for each $f\in \mathcal{A}_{S^1}(M)$, there exists $N\in \mathbb{Z}_{\geqslant 0}$ such that $x^Nf\in \mathcal{H}_{S^1}(M)$.
Let $\{\tau_p\}_{p\in M^{S^1}}$ be the set of canonical classes associated to a fixed orientation of $\fs^*$.
For each $f\in \mathcal{A}_{S^1}(M)$ the sum in \eqref{mc} is an element of $Q(\fs^*)$. However it is clear that, by multiplying $f$ with $x^N$
for $N$ sufficiently large, $$\sum_{p\in M^{S^1}}\frac{x^Nf(p)\tau_q(p)}{e_M(p)}\in \mathbb{S}(\fs^*)\quad\mbox{for every}\;\;q\in M^{S^1}\;,$$
and the conclusion follows from Theorem \ref{milena}.
\end{proof}

\begin{remark}\label{alternative}
It is straightforward to see that Theorem \ref{milena} implies Lemma \ref{lem:bigdim}, thus giving an alternative proof of it.
\end{remark}
The
next theorem is a refinement of Theorem \ref{milena}.
\begin{thm}\label{thm:characterization}
Let $M^{2n}$ be a compact Hamiltonian $S^1$-manifold of dimension $2n$ with discrete 
fixed point set $M^{S^1}$, and $\{\tau_p\}_{p\in M^{S^1}}$ the canonical classes for a fixed 
orientation of $\fs^*$. An assignment $f \in \cA^{2k}_{S^1}(M)$ is cohomological
if and only if 
\begin{equation}\label{eq:condition}
\sum_{p\in M^{S^1}}\frac{f(p)\tau_q(p)}{e_M(p)}=0
\end{equation}
for all $q\in M^{S^1}$\ such that $\lambda_q<n-k$.
\end{thm}

\begin{proof} The proof is based on a dimension count.

If $f$ is a \spass{}, 
then \eqref{eq:condition} is a consequence of the Atiyah-Bott-Berline-Vergne Localization formula.

Suppose now that $f$ is an assignment of degree $2k$ satisfying \eqref{eq:condition}.
Let $q\in M^{S^1}$ be such that $\lambda_q=n\!-\!k\!-\!1$. From the properties of canonical 
classes it follows that
$$
0 = \sum_{p\in M^{S^1}}\frac{f(p)\tau_q(p)}{e_M(p)} = \frac{f(q)\Lambda_q^-}{e_M(q)}+\sum_{\lambda_p\geqslant n-k}\frac{f(p)\tau_q(p)}{e_M(p)}\;.
$$
Hence for all $q\in M^{S^1}$ such that $\lambda_q=n\!-\!k\!-\!1$, $f(q)$ is determined by the values of $f$ at $p$, where $\lambda_p \geqslant n\!-\!k$. 
Inductively, an assignment of degree $2k$ is determined by its values at the fixed points of Morse index greater or equal to $2(n\!-\!k)$, implying that
the space of assignments of degree $2k$ satisfying \eqref{eq:condition} has dimension 
at most
$$
\#\{s\in M^{S^1}\mid \lambda_s\geqslant n-k\} = \#\{s\in M^{S^1}\mid \lambda_s\leqslant k\} = \dim_{\R} H_{S^1}^{2k}(M)
$$
by Poincar\'e duality and \eqref{eq:dimensions}. Since all cohomological assignments 
of degree $2k$ satisfy \eqref{eq:condition}, a dimension count implies that all 
assignments  satisfying \eqref{eq:condition} are cohomological.
\end{proof}

Theorem \ref{thm:characterization} becomes particularly powerful when it is possible 
to compute explicitly  the restrictions $\tau_p(q)$, for every $p,q\in M^{S^1}$. This problem has 
been widely studied, but explicit computations are hard to do in general 
(see \cite{GT, GZ, ST}).  However when $\dim(M)=2n$ and the Hamiltonian action has 
exactly $n\!+\!1$ fixed points, Tolman \cite{T} proved that canonical classes are 
completely characterized by the weights of the isotropy action at the fixed points.

Before describing these canonical classes explicitly we need to recall a few facts. 
Let $(M,\omega)$ be a compact symplectic manifold of dimension $2n$ acted on 
by $S^1$ with isolated fixed points. Let $ES^1$ be a contractible space on which $S^1$ 
acts freely.  Then the total equivariant Chern class of the bundle $TM\to M$ is the 
total Chern class of the bundle $TM\times_{S^1}ES^1\to M\times_{S^1}ES^1$. If $p$ is a fixed point of the $S^1$ action and 
$\{\alpha_1,\ldots,\alpha_n\}\subset \mathfrak{s}^*$ are the weights of
the isotropy representation of $S^1$ on $T_pM$, then the total equivariant Chern class at $p$ is $\prod_{i=1}^n(1+\alpha_i)$. Let  $c_1 = c_1^{S^1}\in H^2_{S^1}(M;\Z)$ be the first equivariant Chern class of the tangent bundle; then $c_1(p)=\alpha_1+\cdots +\alpha_n$, 
and we can rewrite this expression as $c_1(p)=(w_{1p}+\cdots+w_{np})x$, where $w_{ip}\in \Z\setminus \{0\}$
for every $i$ and $x$ is the generator of $H_{S^1}^*(\{p\};\Z)$.

Now assume that the action is Hamiltonian  and has exactly
$n\!+\!1$ fixed points.
Since $[\omega]^k\neq 0$ for all $0\leqslant k\leqslant n$,
from \eqref{eq:dimensions} it's easy to see that  for every $k=0,\ldots,n$
$\#\{s\in M^{S^1}\mid \lambda_s=k\}=1$.
We order the fixed points $p_0,\ldots,p_n$ in such a way that $\lambda_{p_k}=k$ for every $k=0,\ldots,n$.
\begin{remark}
Hamiltonian manifolds of dimension $2n$ with exactly $n+1$ fixed points have the same Betti numbers of $\C P^n$, but are not in 
general diffeomorphic to $\C P^n$. The problem of characterizing all their 
possible cohomology rings and characteristic classes has been introduced 
by Tolman in \cite{T}, where she refers to it as the \emph{symplectic generalization 
of Petrie's conjecture}. Tolman completely analyses the six dimensional case. 
In \cite{GoSa} the authors gave new tools to address this problem in any 
dimension and solved the eight dimensional case. When $M$ is K\"ahler, this 
classification is related to the one of fake projective spaces (see for example \cite{Y}).
\end{remark}

The next lemma summarizes Lemmas 3.8 and 3.23 in \cite{T}.

\begin{lem}[Tolman]
Let $c_1 \in H_{S^1}^2(M;\Z)$ be the first equivariant Chern class. Then
\begin{equation}\label{eq:not_zero}
\frac{c_1(p_i)-c_1(p_j)}{x}>0\quad\mbox{if and only if}\quad i<j\;.
\end{equation}
\end{lem}

The following Theorem combines  Corollaries 3.14 and 3.19 of \cite{T}.

\begin{thm}[Tolman]\label{th:tolman}
Let $M$ be a compact Hamiltonian $S^1$-manifold of dimension $2n$, with fixed point set $\{p_0, \ldots, p_n\}$, such that $\lambda_{p_k}=k$ for all $k=0,\ldots, n$. Then for every $k$ the canonical class $\tau_k=\tau_{p_k}\in H_{S^1}^{2k}(M)$ is given by
\begin{equation}\label{eq:basis_c1}
\tau_k=\frac{1}{C_k}\prod_{j=0}^{k-1}\left(c_1-c_1(p_j)\right),
\end{equation}
where $c_1 \in H^2_{S^1}(M;\Z)$ is the first equivariant Chern class of the tangent bundle and
$$C_k=\frac{1}{\Lambda_k^-}\prod_{j=0}^{k-1}\left(c_1(p_k)-c_1(p_j)\right)\; .$$
\end{thm}

Note that by \eqref{eq:not_zero}, $C_k$ is non-zero for every $k$.

The following is an easy consequence of Theorem \ref{thm:characterization} and Theorem \ref{th:tolman}.

\begin{corollary}\label{cor:min}
Let $M$ be a compact Hamiltonian $S^1$-manifold of dimension $2n$, 
with exactly $n\!+\!1$ fixed points. An assignment $f\in \mathcal{A}_{S^1}^{2k}(M)$ is cohomological if and only if 
\begin{equation}\label{eq:condition2}
\sum_{p\in M^{S^1}}\frac{f(p)(c_1)^i(p)}{e_M(p)}=0\quad\mbox{for all}\quad i<n-k\;.
\end{equation}
\end{corollary}

Corollary \ref{cor:min} gives an explicit characterization of \spass{} for 
Hamiltonian $S^1$-manifolds with minimal number of fixed points. 
We can relax this condition and extend the result by 
combining the Chang-Skjelbred Theorem with the result of Goldin and Holm  
mentioned in Theorem~\ref{thm:GH}. 

The main theorem of this section can be stated as follows.

\begin{thm}\label{th:mainsection}
Let $T$ be a torus and $M$ a compact Hamiltonian $T$-manifold with 
isolated fixed points. Suppose that for each codimension one subtorus 
$K\subset T$, the closure $X$ of each connected component
of $M^K$ satisfies either
\begin{itemize}
\item[(i)] $\dim(X)\leqslant 4$,$\quad$ or
\item[(ii)] the number of fixed points of the $T$ action on $X$ is $d+1$, where $\dim(X)=2d$.
\end{itemize} 
An assignment $f \in \Aktm{2k}{T}{M}$ is cohomological if and only if 
for each codimension one subtorus $K\subset T$, and each connected component $X$ of $M^K$, either
\begin{itemize}
\item[(i)'] $\displaystyle\sum_{p\in X^T}\displaystyle\frac{f(p)}{e_X(p)}\in \mathbb{S}(\mathfrak{t}^*),\quad$ if $\dim(X)\leqslant 4\quad$ or\\$\;$\\
\item[(ii)']$\displaystyle\sum_{p\in X^T}\displaystyle\frac{f(p)(c_{1}^{X})^i(p)}{e_X(p)}=0\quad$ for all $i<d-k,\quad$ if $\dim(X)=2d$ , $|X^T|=d+1.$
\end{itemize} 

Here $c_{1}^X$ denotes the first $T$-equivariant Chern class
of the tangent bundle $TX$.
\end{thm}

\begin{proof}
If $f$ is a cohomological assignment of degree $2k$, then (i)' and (ii)' are consequences of the Atiyah-Bott-Berline-Vergne Localization formula.

Conversely, we prove that if (i)' and (ii)' hold then $f$ is cohomological. 
By Theorem \ref{th:CS_theorem} 
it is sufficient to determine $\bigcap_{j=1}^N i^*_{K_j}(H_T^*(M^{K_j}))$, where $K_1,\ldots,K_N$ denote the subtori of codimension
one which occur as isotropy groups of points of $M$. Note that the $M^{K_j}$s are acted effectively only by the circle $T/K_j$, and
by \eqref{eq:equiv_cohom} it is sufficient to characterize the equivariant cohomology groups $H_{T/K_j}^*(M^{K_j})$.

For every $j=1,\ldots,N$, let $T=S^1_j\times K_j$ be a choice of a splitting, so the manifolds
$M^{K_j}$ are Hamiltonian $S^1_j$-manifolds with isolated fixed points.
Let $X$ be the closure of a connected component of one of the $M^{K_j}$s, and let $f$ be an assignment of degree $2k$.
If  $\dim(X)=4$, then Theorem \ref{thm:GH} implies that $f$ is cohomological. If $f$ satisfies (ii)' then Corollary \ref{cor:min} implies that $f$ is cohomological.
\end{proof}

\section{Functoriality}
\label{sec:kirwan}

\subsection{Pull-backs}
Let $M$ and $N$ be smooth $T$-manifolds, $\cP_T(M)$ and $\cP_T(N)$ the 
corresponding posets of infinitesimal strata, and $f\colon M \to N$ a smooth $T$-equivariant map. 
Then $f$ induces a monotone mapping of posets 
$\tilde{f} \colon \cP_T(M) \to \cP_T(N)$ as follows. 
If $X \in \cP_T^r(M)$ is an infinitesimal stratum of $M$, 
let $Y=\tilde{f}(X) \in \cP_T^r(N)$ be the open dense stratum in the 
connected component of $N^{T_X}$ containing $f(X)$ (see \cite{GGK}). 
Then $\ft_X \subseteq \ft_Y$ and 
$f(X) \subset \overline{Y}$. This map induces a a degree-preserving 
pull-back map $f^* \colon \cA_T(N) \to \cA_T(M)$ by
\begin{equation}\label{eq:pullback}
(f^*A)(X) = \pi^{\tilde{f}(X)}_X A(\tilde{f}(X))\; ,
\end{equation}
for all $A \in \cA_T(N)$ and $X \in \cP_T(M)$.

\begin{lem}\label{functorial}
Let $M$ and $N$ be smooth $T$-manifolds and $f\colon M \to N$ a  smooth $T$-equivariant map. Then 
\begin{equation}
\begin{CD}
H_T^{even}(N) @> f^* >> H_T^{even}(M)    \\
@V \gamma_N VV         @V \gamma_M VV     \\
 \cA_T(N) @> f^* >>  \cA_T(M)
\end{CD} \; 
\end{equation}
is a commutative diagram. Consequently, the pull-back of a cohomological assignment on $N$ is a cohomological assignment on $M$.
\end{lem}

\begin{proof}
Let 
\begin{equation*}
\phi = \phi_0 + \dotsb + \phi_n \colon \ft \to \Omega(N)^T
\end{equation*}
be an equivariant form on $N$, $X \in \cP_T(M)$ an infinitesimal stratum of $M$, and $p \in X$ a point. Then
\begin{equation*}
((f^* \circ \gamma_N) [\phi] )(X)  =  \pi^{\tilde{f}(X)}_X \gamma_N([\phi]) (\tilde{f}(X))
=\pi^{\tilde{f}(X)}_X ( \left. \phi_0 (f(p))\right|_{\ft_{\tilde{f}(X)}} ) =  \left. \phi_0(f(p)) \right|_{\ft_X} 
\end{equation*}
and
\begin{equation*}
((\gamma_M \circ f^*) [\phi] )(X)  = \left. (f^* \phi_0)  (p) \right|_{\ft_X}  
= \left. \phi_0(f(p)) \right|_{\ft_X}  = ((f^* \circ \gamma_N) [\phi] )(X) \; ,
\end{equation*}
and that proves the commutativity of the diagram.
\end{proof}

\subsection{Quotients}
Let $M$ be a $T$-manifold and $Q$ a subtorus of $T$ acting freely on $M$.
Then $M/Q$ has a residual $T/Q$-action, and the projection 
$\pi\colon M \to M/Q$ induces a map
$$\pi_* \colon \cA_T^*(M) \to \cA_{T/Q}^*(M/Q)$$
as follows.

Let  $X \in \cP_{T/Q}(M/Q)$ be an infinitesimal stratum in $M/Q$ and $(\ft/\fq)_X \subseteq \ft/\fq$ the infinitesimal stabilizer of $X$. 
Then $(\ft/\fq)_X = \fk_X/\fq$ for some subtorus $K_X \subset T$ containing $Q$, and $K_X$ fixes all the points in $\overline{X}$.

\begin{lem}
There exists a unique group morphism $\rho_X \colon K_X \to Q$ such that
$$\rho_X (a) \cdot p = a\cdot p$$
for every $a \in K_X$ and $p \in \pi^{-1}(\overline{X})$, and this morphism is surjective.
\end{lem}

\begin{proof}
The uniqueness follows immediately from the fact that $Q$ acts freely on $M$.

If $x \in \overline{X}$, then the fiber $\pi^{-1}(x)$ is $K_X$-invariant.  Let $p \in \pi^{-1}(x)$; then for every $a\in K_X$ there exists a unique 
$b\in Q$ such that $a\cdot p = b \cdot p$, so  we can define a surjective map $\rho_p \colon K_X \to Q$ by the condition
$$\rho_p(a) \cdot p = a\cdot p$$
for all $a \in K$. Moreover, $\rho_p(1) = 1$ and 
$$\rho_p(ab)\cdot p = ab\cdot p = a\rho_p(b)\cdot p = \rho_p(b)a\cdot p = \rho_p(b)\rho_p(a)\cdot p = \rho_p(a)\rho_p(b)\cdot p\; ,$$
hence $\rho_p\colon K_X \to Q$ is a surjective morphism of groups. If $q\in \pi^{-1}(x)$ is another point in the fiber over $x$, then $q=b\cdot p$ 
for some $b \in Q$. If $a\in K_X$, then 
$$b\rho_p(a) \cdot p = ab\cdot p = a\cdot q = \rho_q(a)\cdot q = \rho_q(a)b \cdot p\; ,$$
hence $\rho_p(a) = \rho_q(a)$ for all $a$. The morphism $\rho_p$ does not depend on the chosen point in the fiber $\pi^{-1}(x)$ and 
we can define $\rho_x \colon K_X \to Q$, $\rho_x =\rho_p$ for $p \in \pi^{-1}(x)$. 
Since $X$ is connected, the morphism $\rho_x$ does not depend on the chosen point $x \in \overline{X}$, and we define $\rho_X = \rho_x$ for $x \in X$.
The morphism $\rho_X$ is obviously surjective, since $\rho_X(a) = a$ for all $a \in Q$.
\end{proof}

Let $T_X = \ker \rho_X$ and $\ft_X$ the Lie algebra of $T_X$. Then
$\pi^{-1}(X)$ is included in one of the connected components of  $M^{T_X}$. Let $Z=Z_X$ be the open dense infinitesimal stratum of that component,
and $\ft_Z$ its infinitesimal stabilizer. Then $\ft_X \subset \ft_Z \subset \fk_X$ and the inclusion $\ft_X \subset \ft_Z$ induces a morphism 
$$\pi_{\ft_X}^{\ft_Z} \colon \S(\ft_Z^*) \to \S(\ft_X^*)\; .$$
The composition 
$$T_X \to K_X \to K_X/Q$$
is an isomorphism between $T_X$ and $K_X/Q= (T/Q)_X$; it induces an isomorphism $\varphi_X \colon (\ft/\fq)_X \to \ft_X$ and therefore an isomorphism 
$$\varphi_X^* \colon \S(\ft^*_X) \to \S((\ft/\fq)_X^*)\; .$$

The map $\pi_*$ is defined by
\begin{equation}
\pi_*A(X) = \varphi_X^* \pi_{\ft_X}^{\ft_Z}  (A(Z)) = (i_{\ft_X}^{\ft_Z} \circ \varphi_X)^* (A(Z))\; 
\end{equation}
for all assignments $A \in \cA_T(M)$ and infinitesimal stratum $X \in \cP_{T/Q} (M/Q)$.

\begin{lem}
If $A \in \cA_T(M)$, then $\pi_* A \in \cA_{T/Q}(M/Q)$.
\end{lem}

\begin{proof}
We show that $\pi_* A$ satisfies the compatibility conditions for inclusions.

Let $X \preccurlyeq Y$ be two infinitesimal strata for $M/Q$. Then $(\ft/\fq)_Y \subseteq (\ft/\fq)_X = \fk_X/q$, hence 
$(\ft/\fq)_Y = \fk_Y/\fq$, for some $\fq \subseteq \fk_Y \subseteq \fk_X$. The uniqueness of the morphisms 
$\rho_X$ and $\rho_Y$ implies that $\rho_Y$ is the restriction of $\rho_X$ to $K_Y$ and then $T_Y = T_X \cap K_Y$. 
Therefore $M^{\ft_X} \subseteq M^{\ft_Y}$, which implies $Z_X \subseteq \overline{Z_Y}$, so $Z_X \preccurlyeq Z_Y$ in $\cP_T(M)$ and 
$\ft_{Z_Y} \subseteq \ft_{Z_X}$. The commutative diagram 
$$
\begin{CD}
(\ft/\fq)_Y @>\varphi_Y>> \ft_Y @>>> \ft_{Z_Y}   \\
@VVV         @VVV                  @VVV          \\
(\ft/\fq)_X @>\varphi_X >> \ft_X @>>> \ft_{Z_X}  
\end{CD} \; 
$$
implies that 
$$\pi_Y^X ((\pi_*A)(X)) = (\pi_*A) (Y)\; ,$$
hence $\pi_*A$ is a polynomial assignment on the quotient space $M/Q$. 
\end{proof}

\begin{thm}\label{pi*}
The map 
$$\pi_* \colon \cA_T^*(M) \to  \cA_{T/Q}^{*}(M/Q)$$
is an isomorphism.
\end{thm}

\begin{proof} The proof is based on the construction of an inverse map
$$\sigma \colon \cA_{T/Q}^*(M/Q) \to \cA_T^{*}(M) $$
as follows: Let $A \in \cA_{T/Q}^*(M/Q)$ and $X \in \mathcal{P}_T^{r}(M)$ a degree-$r$ infinitesimal orbit-type 
stratum, with infinitesimal stabilizer $\ft_X$ of codimension $r$ in $\ft$. For $p \in X$, let $K_p$ be the subtorus of $T$ that stabilizes the $Q$-fiber through $p$. Let $\rho_p \colon K_p \to Q$ be the map defined by
\begin{equation*}
a\cdot p = \rho_p(a) \cdot p
\end{equation*}
for all $a \in K_p$. Then $\rho_p$ is a surjective morphism of groups and the kernel of $\rho_p$ 
is the stabilizer $T_X = \exp(\ft_X)$ of $p$. Because $X$ is connected, both $K_p$ and $\rho_p$ are independent of $p \in X$, 
hence we have a subtorus $K_X$ of $T$ and a surjective group morphism $\rho_X \colon K_X \to Q$ such that 
\begin{equation*}
a \cdot p = \rho_X(a) \cdot p
\end{equation*}
for every $a \in K_X$ and $p \in X$. The kernel of this morphism is $T_X$ and the corresponding infinitesimal exact sequence
\begin{equation*}
\ft_X \to \fk_X \to \fq
\end{equation*}
defines an isomorphism $\psi_X \colon \ft_X \to \fk_X/\fq$. Moreover, $\fk_X = \ft_X \oplus \fq$. To see that, note that both $\ft_X$ and $\fq$ are included in $\fk_X$, hence their sum is, too. The sum is direct because $Q$ acts freely and $T_X$ stabilizes points in $X$, hence $\ft_X \cap \fq = \{ 0 \}$. The direct sum has the same dimension as $k_X$ and is included in $\fk_X$, hence it is equal to $\fk_X$.

The infinitesimal stabilizer of points in $X/Q$ is $\fk_X/\fq$; let $\widetilde{X}$ be the open dense stratum in the connected component of $(M/Q)^{K_X/Q}$ containing $X/Q$. We define
\begin{equation}
\sigma(A) (X) = \psi_X^* A(\widetilde{X})\; .
\end{equation}
If $X \preccurlyeq Y$, then the diagram
$$
\begin{CD}
\ft_Y @>\psi_Y >> \fk_Y/\fq  \\
@VVV         @VVV                    \\
\ft_X @>\psi_X >> \fk_X/\fq 
\end{CD} \; 
$$
commutes and $\widetilde{X} \preccurlyeq \widetilde{Y}$. Therefore
\begin{equation*}
\sigma(A)(Y) = \psi_Y^* (A(\widetilde{Y})) = \psi_Y^* \pi_{\widetilde{Y}}^{\widetilde{X}} A(\widetilde{X}) = \pi_Y^X \psi_X^* A(\widetilde{X}) = \pi_Y^X \sigma(A)(X)
\end{equation*}
hence $\sigma(A) \in  \cA_T^{*}(M)$.

The compositions $\sigma \circ \pi_*$ and $\pi_* \circ \sigma$ are identities, 
hence $\sigma$ is the inverse of $\pi_*$.
\end{proof}

\subsection{Kirwan map for assignments}
Let $M$ be a Hamiltonian $T$-space with moment map $\Phi \colon M \to \ft^*$ and $Q\subset T$ a subtorus of $T$. 
Then $M$ is also a 
Hamiltonian $Q$-space, with induced moment map $\Phi_Q \colon M \to \fq^*$. 
Suppose 0 is a regular value of $\Phi_Q$ and $Q$ acts 
freely on $M_0 = \Phi_Q^{-1}(0)$. The reduced space $M_{red} = M_0/Q$ is 
a $T/Q$-Hamiltonian space. If $i \colon M_0 \to M$ is the inclusion and 
$\pi \colon M_0 \to M_{red} = M_0/Q$ the quotient map, then the composition
$$\cA_T(M) \xrightarrow{i^*} \cA_T(M_0) \xrightarrow{\pi_*} \cA_{T/Q}(M_{red})$$
defines a canonical map
$$\cK_{\cA} = \pi_* i^* \colon \cA_T(M) \to \cA_{T/Q}(M_{red}) \; .$$

The map $\cK_{\cA}$ is the assignment version of the Kirwan map 
$$\cK_H \colon H_T^*(M) \xrightarrow{i^*} H_T^*(M_0) 
\xrightarrow{\pi_*} H_{T/Q}^*(M_{red})$$
in equivariant cohomology.

\def\cprime{$'$}

\end{document}